\documentclass[11pt,a4paper]{article}
\usepackage{amsmath,amsthm,amsopn,amsfonts,graphicx,amssymb,dsfont,color}
\usepackage{mathrsfs}
\usepackage{color}
\usepackage{hyperref}

\usepackage{enumerate}
\usepackage{a4,a4wide}

\def\ep{{\varepsilon}}
\def\R{\mathbb R}
\def\N{\mathbb N}

\def\P{\mathcal P}
\def\A{\mathcal A}

\newcommand{\ds}{\displaystyle}

\newtheorem{theo}{\textbf{Theorem}}[section]

\newtheorem{lem}[theo]{\textbf{Lemma}}

\newtheorem{cor}[theo]{\textbf{Corollary}}

\newtheorem{defi}[theo]{\textbf{Definition}}

\newtheorem{assumption}[theo]{\textbf{Assumption}}
\newtheorem{rem}[theo]{\textbf{Remark}}

\title{\bf The Allen-Cahn equation with nonlinear truncated Laplacians: description of radial solutions}
\date{}

\begin{document}
\maketitle

\begin{center}
{\large\bf Matthieu Alfaro\footnote{Univ. Rouen Normandie, CNRS, LMRS UMR 6085, F-76000 Rouen, France.} and Philippe Jouan\footnote{Univ. Rouen Normandie, CNRS, LMRS UMR 6085, F-76000 Rouen, France.}}
\end{center}


\tableofcontents

\vspace{10pt}

\begin{abstract} We consider the Allen-Cahn equation with the so-called truncated Laplacians, which are fully nonlinear differential operators that depend on {\it some}  eigenvalues of the Hessian matrix. By monitoring the sign of  a quantity that is responsible for switches from a first order ODE regime to a second order ODE regime (and {\it vice versa}), we give a nearly complete description of  radial solutions. In particular, we reveal the existence of surprising unbounded radial solutions. Also radial solutions cannot oscillate, which is in sharp contrast with the case of the Laplacian operator, or that of Pucci's operators.\\


\end{abstract}

\section{Introduction}\label{s:intro}

This work is concerned with the highly degenerate partial differential  equations
\begin{equation}
\label{eq-moins}
\P _k ^{-}  U+g(U)=0 \quad \text{ in }  \R^{N},
\end{equation}
and
\begin{equation}
\label{eq-plus}
\P   _k^{+} U+g(U)=0 \quad \text{ in } \R^{N}.
\end{equation}
Here $N\geq 2$ and, for $1\leq k<N$, $\P _k^{-}$ et $\P _k ^{+}$ denote the so-called nonlinear truncated Laplacian operators. The nonlinearity $g$ is of the bistable type, a typical example being the Allen-Cahn nonlinearity $g(U)=U-U^{3}$. In this framework, we give a nearly complete description of radial solutions, revealing sharp differences with the case of the Laplacian operator, or that of Pucci's operators.

\medskip

The nonlinear truncated Laplacians are defined as follows. Let $N\geq 2$ be given. For $U:\R^N\to \R$, say of the class $\mathcal C^2$, we denote by
$$
\lambda_1(D^2U)\leq \cdots\leq \lambda _N(D^2U)
$$
the eigenvalues of the Hessian matrix $D^2U$. For $1\leq k<N$ we consider the fully nonlinear differential operators given by
\begin{equation}
\label{Pmoins}
\P _k^- U:=\sum_{i=1}^k \lambda_i(D^2U),
\end{equation}
and
\begin{equation}
\label{Pplus}
\P _k^+ U:=\sum_{i=N-k+1}^N \lambda_i(D^2U).
\end{equation}

These highly degenerate operators have been introduced in the context of differential 
geometry, by Wu \cite{WuH} and Sha \cite{ShaJP}, in order to solve problems related to manifolds with partial 
positive curvature. They also appear in the analysis of mean curvature flow in arbitrary 
codimension performed by Ambrosio and Soner   \cite{Amb-Son-96}.

In a PDE context, they are considered as an example of degenerate fully non linear operators 
in the seminal {\it User's guide} \cite{Cra-Ish-Lio-92}, but more recently both 
Harvey and Lawson in \cite{HarvLawCPAM, HarvLawIUMJ} and Caffarelli, Li and Nirenberg \cite{CafLiNirIII} have 
studied them in a completely new light.

In the very last years, some new results have been obtained on Dirichlet problems in bounded domains in 
relationship with the convexity of the domain, through the study of the maximum principle and the  principal eigenvalue,  see \cite{OberSilv} and  \cite{Bir-Gal-Ish-18, Bir-Gal-Ish-20}. 

As far as problems in $\R ^N$ are concerned, we refer to the works of Birindelli, Galise and Leoni \cite{Bir-Gal-Leo-17}  and  Galise \cite{Gal-19} on the existence of steady states (nonlinear Liouville theorems),  and to the work \cite{Alf-Bir-20} on solutions to evolution equations (Heat equations and Fujita blow up phenomena).

Let us also mention the works  \cite{BlaRos}, \cite{BlaEsteRos} involving the
degenerate  operators defined by $\mathcal P _j U:=\lambda_j(D^2U)$ for some $1\leq j\leq N$, studying  well-posedness of such problems, and their approximation by a two-player zero-sum game.

\medskip

When considering the Laplacian operator, the search of radial solutions $U(x)=u(\vert x\vert)$ to
$$
\Delta U +g(U)=0\quad \text{ in } \R ^N,
$$
reduces to understanding, for any $\xi\in \R$, the non autonomous second order ODE Cauchy problem
$$
u''+\frac{N-1}{r}u'+g(u)=0\quad  \text{ in } (0,+\infty), \quad u(0)=\xi, \quad u'(0)=0,
$$
where $u=u(r)$. Building on this, one can prove, when $g$ is an odd and bistable nonlinearity,  the existence of {\it oscillating} solutions  which, moreover, are periodic when $N=1$ and {\it localized} when $N\geq 2$ (see below for precise definitions).  We refer to \cite{Man-Mon-Pel-17} for such constructions, see also \cite{Gui-Zho-08} and \cite{Gui-Luo-Zho-09} for related results. Such oscillating radial solutions  were also constructed for the bistable prescribed mean curvature equations \cite{Pom-18}, namely
$$
\textrm{div}\left(\frac{\nabla U}{\sqrt{1\pm \vert \nabla U\vert ^2}}\right)+g(U)=0 \quad \text{ in } \R^N,
$$
the sign \lq\lq$+$'' corresponding to the Euclidian case, the sign \lq\lq$-$'' to the Lorentz-Minkowski case.

Very recently, the case of Pucci's extremal operators was studied by d'Avenia and Pomponio \cite{Ave-Pom-20}. In contrast with the Laplacian operator or the mean curvature operator, the Pucci's operators of radial functions may take two different forms depending on the sign of $u''(r)$. Because of that, in order to construct radial solutions, one needs to monitor the quantity $u''(r)$ and may have to glue solutions of two different second order ODE Cauchy problems.  Nevertheless, as proved in \cite{Ave-Pom-20}, oscillating radial solutions still exist.

The nonlinear truncated Laplacian operators share with Pucci's operators the property of switching their expressions depending on the sign of a quantity, namely $u''(r)-\frac 1 r u'(r)$. Tracking such a quantity is far from trivial and, moreover, when a switch occurs, one has to glue a solution of a first order ODE with a solution of a second order ODE. This makes the analysis rather involved and outcomes are in sharp contrast with the aforementioned equations. In particular, it turns out that equations \eqref{eq-moins} and \eqref{eq-plus} support the existence of unexpected unbounded  solutions but, on the other hand, do not admit any oscillating (radial) solutions.

\section{Main results}\label{s:results}

We consider a bistable nonlinearity $g$, a typical example being  the Allen-Cahn nonlinearity $g(U)=U-U^{3}$. More generally, we always assume the following. 

\begin{assumption}[Bistable nonlinearity]\label{ass:nonlinearity} The nonlinearity $g:\R\to \R$ is odd, and of the class $\mathcal{C}^1$ on $\R$. There are $0<\beta<\alpha$ such that
	\begin{enumerate}
		\item [(i)] $g>0$ on $(0,\alpha)$, $g<0$ on $(\alpha,+\infty)$,		
		\item  [(ii)] $g'> 0$ on $[0,\beta)$, $g'< 0$ on $(\beta,+\infty)$,
		\item [(iii)] $g$ is twice differentiable at $\pm\beta$ and $g''(\beta)<0$.
	\end{enumerate}
\end{assumption}

In this work, by (radial) solution to \eqref{eq-moins}, or \eqref{eq-plus},  we always mean the following.

\begin{defi}[Radial solutions]\label{def:radial} A radial solution to \eqref{eq-moins}, or \eqref{eq-plus}, is a $0<R\leq +\infty$ and a piecewise $\mathcal C^2$ function $u:[0,R)\to \R$,  with $u'(0)=0$, such that $U(x):=u(\vert x\vert)$ solves \eqref{eq-moins}, or \eqref{eq-plus}, on $B(0,R)$. Also if $R<+\infty$ we require $\vert u(r)\vert +\vert u'(r)\vert \to +\infty$ as $r\nearrow R$, meaning that the solution is  maximal.
	
	It is said oscillating if $R=+\infty$ and $u$, nontrivial, has an unbounded sequence of zeros. 
	
	It is said localized if $R=+\infty$ and $u(+\infty)=0$.
\end{defi}

\begin{rem} [Regularity of radial solutions] 
	Following the construction in Section \ref{s:radial}, one can easily see that some solutions $u$ are of the class $\mathcal C^2$ on $[0,R)$, while some are only of the class $\mathcal C^1$ on $[0,R)$, $\mathcal C^2$ on $[0,R)\setminus\{t_{2\rightarrow 1}\}$ for some switching point $t_{2\rightarrow 1}$, with $u''(t_{2\rightarrow 1}^-)>u''(t_{2\rightarrow 1}^+)$. This jump in the second derivative occurs when the operator switches from a second order regime to a first order regime (the other way being harmless). See below for details. 
\end{rem}

\begin{rem}[Considering \eqref{eq-plus} is enough]\label{rem:enough}
Since $\P _k ^- U=-\P _k ^+(-U)$ and $g$ is odd, we deduce that $U$ solves \eqref{eq-moins} if and only if $-U$ solves \eqref{eq-plus}. As a result, we focus only on  equation \eqref{eq-plus}, but any of our stated result has its immediate counterpart for equation \eqref{eq-moins}.
\end{rem}

In order to state our main results, we need to define the following critical value: from Assumption \ref{ass:nonlinearity} there is a unique $\xi^{*}=\xi ^*(k)\in(\beta,\alpha)$ such that
\begin{equation}\label{def:xi-etoile}
kg^{2}(\beta)\int_\beta ^ {\xi^{*}} \frac{ds}{g(s)}=\int _\beta ^{\alpha}g(s)ds.
\end{equation}


We now state our main result for $k=1$, a case for which our description of radial solutions is exhaustive. 

\begin{theo}[Radial solutions, $k=1$]\label{th:radial-k-1} Let $k=1$. Then for any $\xi \in \R$, there is a unique radial solution to \eqref{eq-plus} starting from $u(0)=\xi$. Moreover, depending on $\xi$, it has the following properties.
\begin{enumerate}
\item [(i)] If $\xi\in\{-\alpha,0,\alpha\}$ then $u$ is constant (and $R=+\infty$).
\item [(ii)] If $\xi<-\alpha$ then  $u'<0$ on $(0,R)$, $u(r)\to -\infty$ as $r\nearrow R$. 
\item [(iii)] If $\xi>\alpha$ then  $u'>0$ on $(0,R)$, $u(r)\to +\infty$ as $r\nearrow R$. 
\item [(iv)] If $-\alpha < \xi<0$ then $R=+\infty$, $u'>0$ on $(0,+\infty)$, and $u(+\infty)=0$. 
\item [(v)] If $0 <\xi <\alpha$ then the three following outcomes are possible.
\begin{enumerate}
\item [(a)]  If $\xi^*<\xi <\alpha$ then $u'<0$ on $(0,R)$, $u(r)\to -\infty$ as $r\nearrow R$. 

\item [(b)] If $\xi=\xi^*$ then $R=+\infty$, $u'<0$ on $(0,+\infty)$, $u(+\infty)=-\alpha$. 

\item [(c)] If $0<\xi<\xi^*$ then $R=+\infty$, there is $b>0$ such that $u'<0$ on $(0,b)$, $u(b)<0$, $u'>0$ on $(b,+\infty)$, and $u(+\infty)=0$. 
\end{enumerate}
\end{enumerate}
\end{theo}

Hence depending on the initial value $\xi$, radial solutions can be: $(i)$ constant; $(ii)$ negative, decreasing and unbounded; $(iii)$ positive, increasing and unbounded; $(iv)$ negative, increasing and localized; $(v)$-$(a)$ sign changing, decreasing and unbounded; $(v)$-$(b)$ sign-changing, decreasing and bounded; $(v)$-$(c)$ sign-changing, decreasing-increasing and localized. Note that, when $k=1$, equations \eqref{eq-moins} and \eqref{eq-plus} do not admit periodic solutions, which is in sharp contrast with the case of the Laplacian \cite{Man-Mon-Pel-17} and of the Pucci's operators \cite{Ave-Pom-20}. 

As revealed below in \eqref{calcul-Pplus}, if $u''(r)-\frac 1 r u'(r)\geq 0$ then $\P _1^{+}U(x)=u''(r)$,  corresponding to the Laplacian in dimension $k=1$. Nevertheless, for $\vert \xi \vert <\alpha$, the behaviour of radial solutions to \eqref{eq-plus} is very different from those to the bistable equation involving the Laplacian operator in dimension one, namely
\begin{equation}
 v'' +g(v)=0, \quad v(0)=\xi \in(-\alpha,\alpha),\quad  v'(0)=0,
\end{equation}
which are all bounded, in sharp contrast with the case $(v)$-$(a)$ of Theorem \ref{th:radial-k-1}, sign-changing, in sharp contrast with the case $(iv)$ of Theorem \ref{th:radial-k-1}, and periodic, in sharp contrast with the cases $(iv)$ and $(v)$ of Theorem \ref{th:radial-k-1}. The reason is that, as revealed below in \eqref{calcul-Pplus}, if $u''(r)-\frac 1 r u'(r)\leq 0$ then $\P _1^{+}U(x)=\frac 1 r u'(r)$. As a result, before entering the \lq\lq Laplacian in dimension one regime'', the energy of the solution may be increased by a \lq\lq first order ODE regime'', explaining $(v)$-$(a)$ and $(v)$-$(b)$. Also, the solution always finishes his way by the \lq\lq first order ODE regime'', which prevents oscillations that one may have expected, for instance in the case $(v)$-$(c)$. 

\medskip

We now state our main result for $k\geq 2$.  

\begin{theo}[Radial solutions, $k\geq 2$]\label{th:radial-k-2} Let $k\geq 2$. Then for any $\xi \in \R$, there is a unique radial solution to \eqref{eq-plus} starting from $u(0)=\xi$. Moreover, depending on $\xi$, it has the following properties.
\begin{enumerate}
\item [(i)] If $\xi\in\{-\alpha,0,\alpha\}$ then $u$ is constant (and $R=+\infty$).
\item [(ii)] If $\xi<-\alpha$ then  $u'<0$ on $(0,R)$, $u(r)\to -\infty$ as $r\nearrow R$. 
\item [(iii)] If $\xi>\alpha$ then  $u'>0$ on $(0,R)$, $u(r)\to +\infty$ as $r\nearrow R$. 
\item [(iv)] If $-\alpha < \xi<0$ then $R=+\infty$, $u'>0$ on $(0,+\infty)$, and $u(+\infty)=0$. 
\item [(v)] If $0 <\xi <\alpha$ then there are positive numbers $\underline \xi$, $\xi^{**}$, $\overline \xi$ satisfying $\beta<\underline{\xi}\leq \xi^{**}\leq \overline \xi<\alpha$  such that the 
three following outcomes are possible.
\begin{enumerate}
\item [(a)]  If $\overline \xi<\xi <\alpha$ then $u'<0$ on $(0,R)$, $u(r)\to -\infty$ as $r\nearrow R$. 

\item [(b)] If $\xi=\xi^{**}$ then $R=+\infty$, $u'<0$ on $(0,+\infty)$, $u(+\infty)=-\alpha$. 

\item [(c)]  If  $0<\xi<\underline \xi$ then $R=+\infty$, there is $b>0$ such that $u'<0$ on $(0,b)$, $u(b)<0$, $u'>0$ on $(b,+\infty)$, and $u(+\infty)=0$. 
\end{enumerate}
\end{enumerate}
\end{theo}

We conjecture that, in the setting of Theorem \ref{th:radial-k-2}, we  may take $\underline \xi=\xi^{**}=\overline \xi$,  so that the above description would be exhaustive. We refer to Theorem \ref{th:ODE2-bis} and Remark \ref{rem:not-exhaustive} for further details and comments on this delicate issue.

Let us underline that, using the one-to-one relation \eqref{r-zero} (see also subsection \ref{ss:switch}) and the lower bound in \eqref{tous-les-r}, one can check that $\underline \xi>\xi^*$, where $\xi^*=\xi^*(k)$ is defined through \eqref{def:xi-etoile}. Roughly speaking, this means that there are \lq\lq more'' localized solutions, of the type $(v)$-$(c)$, in the $k\geq 2$ case than in the $k=1$ case. The reason is the following.  As revealed below in \eqref{calcul-Pplus}, if $u''(r)-\frac 1 ru'(r)\geq 0$ then $\P _k ^+U(x)=u''(r)+\frac{k-1}{r}u'(r)$, corresponding to the Laplacian in dimension $k$.  In this regime, we therefore have to deal with a non autonomous second order ODE when $k\geq 2$, which typically dampens oscillations.  Consequently, the possible increase of the energy through a \lq\lq first order regime'' --- when $u''(r)-\frac 1 ru'(r)\leq 0$--- is less sensitive when $k\geq 2$. 

Note that, when $k\geq 2$, equations \eqref{eq-moins} and \eqref{eq-plus} do not admit oscillating solutions, which is again in sharp contrast with the case of the Laplacian \cite{Man-Mon-Pel-17} and of the Pucci's operators \cite{Ave-Pom-20}.

\medskip

The organization of the paper is as follows. Some preliminaries are presented in Section \ref{s:preli}: identification of two distinct regimes for the computation of the nonlinear truncated Laplacians of radial functions, and analysis of the needed ODE Cauchy problems. In Section \ref{s:radial}, we prove Theorem \ref{th:radial-k-1} and Theorem \ref{th:radial-k-2}, thus giving a complete description of radial solutions when $k=1$, and a nearly complete one when $k\geq 2$.

\section{Preliminaries}\label{s:preli}

\subsection{Truncated Laplacians of radial functions}\label{ss:radial}

Consider a  radial function
\begin{equation}
\label{ansatz}
U(x):=u\left(\vert x\vert\right), 
\end{equation}
for some $u=u(r)$, and where $\vert x\vert =(x_1^2+\cdots+x_N^2)^{1/2}$. We also require $u'(0)=0$.

After straightforward computations, we obtain the Hessian matrix
$$
D^{2}U(x)=\frac{1}{\vert x\vert}u'(r)Id_N-\left(\frac{1}{\vert x\vert}u'(r) -u''(r) \right)\frac{x}{\vert x\vert }\otimes\frac{x}{\vert x\vert},\quad  \text{ for all } x\neq 0,
$$
where we understand $\vert x\vert =r$. Since $\frac{x}{\vert x\vert}\otimes \frac{x}{\vert x\vert}$ is a matrix of rank 1, $\frac{1}{\vert x\vert}u '(r)$ is an eigenvalue of $D^{2}U(x)$ with multiplicity (at least) $N-1$. By considering the trace of the Hessian matrix we see that the remaining eigenvalue has to be $u ''(r)$. As a result, comparing $\frac{1}{\vert x\vert}u '(r)$ and $u''(r)$ is enough to compute $\mathcal P^{\pm} _k U(x)$. More precisely, denoting
\begin{equation}
\label{def:A}
\A u(r):=u''(r)-\frac 1 r u'(r),
\end{equation}
we have
\begin{equation}\label{calcul-Pmoins}
\P ^{-}_kU(x)=\begin{cases}\frac{k}{r}u'(r) &\text{ if } \A u(r)\geq 0 \vspace{5pt}\\
u''(r)+\frac{k-1}{r}u'(r) &\text{ if } \A u(r)\leq 0,
\end{cases}
\end{equation}
whereas
\begin{equation}\label{calcul-Pplus}
\P ^{+}_kU(x)=\begin{cases}u''(r)+\frac{k-1}{r}u'(r) &\text{ if }  \A u(r)\geq 0 \vspace{5pt}\\
\frac k r u'(r) &\text{ if } \A u(r)\leq 0.
\end{cases}
\end{equation}

As a result, when looking after radial solutions to \eqref{eq-moins} or \eqref{eq-plus}, we have to deal either with a first order ODE, namely
$$\frac k r u'+g(u)=0,
$$
or a second order ODE, namely 
$$
u''+\frac{k-1}{r}u'+g(u)=0,
$$
depending on the sign of $\A u$, and switches from a regime to another may occur. 

\subsection{Some first order ODE tools}\label{ss:ode-ordre-1}

For $r_0\geq 0$ and $\xi \in \R$, we consider  the  nonlinear first order ODE  Cauchy problem
\begin{equation}\label{cauchy-ordre1}
\left\{\begin{array}{lllll} \varphi'&=& -\frac r k g(\varphi),\\
\varphi(r_0) &=& \xi,
\end{array}\right.
\end{equation}
and denote $\varphi(r)=\varphi(r;r_0,\xi)$ its maximal solution defined on some open interval $I$ containing $r_0$. 

\begin{lem}[First order ODE Cauchy problem]\label{lem:ODE1}  If $\vert \xi \vert \leq \alpha$ then the solution $\varphi$ is global.  Next, depending on the initial data $\xi$, the following holds.
\begin{enumerate}
\item [(i)]  If $\xi\in \{-\alpha,0,+\alpha\}$ then $\varphi$ is constant.

\item [(ii)] If $0<\xi<+\alpha$ then  $\varphi '<0$ on $(r_0,+\infty)$, and $\varphi(r)\to 0$ as $r\to +\infty$. 

If $-\alpha <\xi<0$  then  $\varphi '>0$ on $(r_0,+\infty)$, and $\varphi(r)\to 0$ as $r\to +\infty$.

\item [(iii)]  If $\xi>+\alpha$ then $\varphi '>0$ on $I\cap (r_0,+\infty)$, and $\varphi(r)\to +\infty$ as $r\to \sup I$.

If $\xi<-\alpha$ then $\varphi '<0$ on $I\cap (r_0,+\infty)$, and $\varphi(r)\to -\infty$ as $r\to \sup I$.
\end{enumerate}
Last, in any case, we have
\begin{eqnarray}
\A \varphi (r)&=&-\frac {r} k \varphi'(r)g'(\varphi(r))=\frac {r^2}{k^2} g(\varphi(r))g'(\varphi(r)), \quad \forall r\in I, \label{A-phi}\\
(\A \varphi)' (r)&=& \left(\frac 2 r-\frac r kg'(\varphi(r))\right)\A \varphi (r)-\frac r k {\varphi'}^{2}(r)g''(\varphi(r)), \quad \forall r \in I\setminus\{0\}, \label{A-phi-prime}
\end{eqnarray}
the latter obviously holding true when $g$ is twice differentiable at $\varphi(r)$ (in particular when $\varphi(r)=\pm\beta$).
\end{lem}

\begin{proof}  Item $(i)$ is clear since $-\alpha$, $0$ and $\alpha$ are the three zeros of $g$.
 
Let us prove $(ii)$. Assume $0<\xi<\alpha$, the other case being similar. Then, from Cauchy-Lipschitz theorem, $\varphi(r)$ is trapped  in $(0,\alpha)$ so that not only $I=\R$, but also $\varphi '<0$ on $(r_0,+\infty)$ from the ODE and Assumption \ref{ass:nonlinearity}.  Denote
\begin{equation}
\label{def:H}
H(t):=\int_t^\xi \frac{ds}{g(s)}, \quad 0<t\leq \xi.
\end{equation}
From Assumption \ref{ass:nonlinearity}, $H$ is a decreasing bijection from $(0,\xi)$ to $(0,+\infty)$. By separating variables, we can compute
\begin{equation}
\label{varphi}
\varphi(r)=H^{-1}\left(\frac{r^{2}-r_0^{2}}{2k}\right)\to 0, \quad \text{ as } r\to +\infty.
\end{equation}

Let us prove $(iii)$. Assume $\xi >\alpha$, the other case being similar. Then, from Cauchy-Lipschitz theorem, $\varphi(r)$ is trapped in $(\alpha,+\infty)$ so that  $\varphi '>0$ on $I\cap (r_0,+\infty)$ from the ODE and Assumption \ref{ass:nonlinearity}. If $\sup I<+\infty$ then one must have $\varphi(r)\to +\infty$ as $r\to \sup I$. On the other hand, if $\sup I=+\infty$ then, from the ODE, the limit in $+\infty$ cannot be finite (otherwise $\varphi'(r)\to +\infty$ which is a contradiction), and thus has to be $+\infty$.

Last, recalling $\A \varphi (r)=\varphi''(r)-\frac 1 r\varphi'(r)$, we differentiate once the ODE \eqref{cauchy-ordre1} to reach \eqref{A-phi}, that we differentiate to get \eqref{A-phi-prime}.
\end{proof}

For a solution $\varphi$,  we say that  $\A \varphi$ ceases to be nonpositive at $s_0>0$ 
if $(i)$ there is $\ep_1>0$ such that $\A \varphi \leq 0$ on $(s_0-\ep_1,s_0)$ and  $(ii)$ for all $\ep_2>0$, there is $s^*\in(s_0,s_0+\ep_2)$ such that $\A \varphi (s^*)>0$.

\begin{cor}\label{cor:switch-phi} $\A \varphi$ ceases to be nonpositive at $s_0>0$ if and only if $\varphi(s_0)=\beta$; and, if so, there is $\ep_3>0$ such that $\A\varphi >0$ on $(s_0,s_0+\ep_3)$. 
\end{cor}

\begin{proof}  If $\A\varphi$ ceases to be nonpositive at $s_0>0$ then $\A\varphi(s_0)=0$.  We deduce from \eqref{A-phi} that $g(\varphi(s_0))g'(\varphi(s_0))=0$. If $g(\varphi(s_0))=0$ then $\varphi(s_0)\in\{-\alpha,0,\alpha\}$ and, from the equation for $\varphi$, $\varphi'(s_0)=0$ so that $\varphi$ is constant, which is a contradiction. Hence $g'(\varphi(s_0))=0$, meaning $\varphi(s_0)=\pm\beta$, and, from \eqref{A-phi-prime} and the equation for $\varphi$, we have 
$(\A\varphi)'(s_0)=-\frac{s_0^3}{k^3}g^2(\pm\beta)g''(\pm \beta)$ which, see Assumption \ref{ass:nonlinearity}, is negative at $-\beta$ (hence a contradiction) and positive at $+\beta$. The converse is obviously true since then $\A\varphi(s_0)=0$ and $(\A\varphi)'(s_0)>0$.
\end{proof}

\subsection{Some second order ODE tools}\label{ss:ode-ordre-2}

For $r_0\geq 0$, $\xi \in \R$ and $\theta \in \R$ with
\begin{equation}
(r_0,\theta)\in (0,+\infty)\times \R \cup \{(0,0)\},
\end{equation}
 we consider  the  nonlinear second order ODE  Cauchy problem
\begin{equation}\label{cauchy-ordre2}
\left\{\begin{array}{lllll} \psi''+\frac  {k-1}r \psi'+g(\psi)&=&0,\\
\psi(r_0) &=& \xi,\\
\psi'(r_0)&=&\theta,
\end{array}\right.
\end{equation}
and denote by $\psi(r)=\psi(r;r_0,\xi,\theta)$ its maximal solution defined on some open interval $J$ containing $r_0$.

\begin{rem}[Sturm-Liouville approach]\label{rem:Sturm} {\it Stricto sensu}, when $k\neq 1$  and $(r_0,\theta)=(0,0)$, the Cauchy-Lipschitz theorem does not apply to \eqref{cauchy-ordre2}. However using a Sturm-Liouville approach, one can 
write $(r^{k-1}\psi')'=-r^{k-1}g(\psi(r))$ so that
\begin{equation}\label{psi-prime}
\psi'(r)=-\frac{1}{r^{k-1}}\int_0^r s^{k-1}g(\psi(s))ds,
\end{equation}
and then recast the ODE problem \eqref{cauchy-ordre2} into the integral equation
$$
\psi(r)=\xi -\int_0^r \frac{1}{s^{k-1}}
\int_0^s t^{k-1}g(\psi(t))dt ds.
$$
Under this form one can then prove that the conclusion of the Cauchy-Lipschitz theorem does hold, see \cite[Proposition 3.1]{Har-Wei-82} among others.

Furthermore, when $k\geq 2$, if $\psi$ solves the second order ODE on $(0,R)$ and can be extended by continuity at $r=0$, say $\psi(0)=\xi \in \R$, then, necessarily, $\psi'(0)=0$ and $\psi''(0)=-\frac{g(\xi)}{k}$ 
(notice that  the solution $\varphi$ to the first order Cauchy problem \eqref{cauchy-ordre1} also obviously satisfies $\varphi'(0)=0$ and $\varphi''(0)=-\frac{g(\xi)}{k}$).  Indeed, the Sturm-Liouville approach provides, for any $0<r_0<r<R$, 
$$
		r^{k-1}\psi'(r)-r_0^{k-1}\psi'(r_0)=-\int_{r_0}^{r}s^{k-1}g(\psi(s))ds.
$$
The right hand side has a finite limit as $r_0\to 0$ and, thus, so has $r_0^{k-1}\psi'(r_0)$; since $k\geq 2$ and $\psi(0)=\xi$, this limit has to be zero. As a result \eqref{psi-prime} is still valid and can be recast
$$
\psi'(r)=-r \int_0^1 z^{k-1}g(\psi(rz))dz,
$$
so that $\psi'(0)=0$ and $\frac{\psi'(r)}{r} \to -\frac{g(\xi)}{k}$  as $r \to 0$, as announced. In particular, when $k\geq 2$, the requirement $u'(0)=0$ is automatically satisfied by solutions in the sense of Definition \ref{def:radial}  (and also by all solutions that can be extended to zero, see subsection \ref{ss:final}).
\end{rem}

\begin{lem}[Monitoring $\A \psi$]\label{lem:monitoring} We have, for all $r\in J\setminus \{0\}$,
\begin{eqnarray}
\A \psi (r)&=&-\frac k r \psi'(r)-g(\psi(r)),\label{A-psi}\\
(\A \psi)'(r)&=&-\frac k r \A \psi(r)-\psi'(r)g'(\psi(r)), \label{A-psi-prime}\\
(\A \psi)''(r)&=&\frac{k}{r^{2}}\A \psi(r)-\frac k r (\A \psi)'(r)-\psi''(r)g'(\psi(r))-{\psi'}^{2}(r)g''(\psi(r)),\label{A-psi-prime-prime}
\end{eqnarray}
the latter obviously holding true when $g$ is twice differentiable at $\psi(r)$  (in particular when $\psi(r)=\pm\beta$). Also, for any $r_0\in J$,
	\begin{equation}\label{A-psi-integrated}
		\A\psi(r)=\A\psi (r_0)\frac{r_0^k}{r^{k}}-\int_{r_0}^{r}\frac{s^k}{r^k}\psi'(s)g'(\psi(s))ds,\quad 	\forall r\in J\setminus \{0\}. 
	\end{equation}
\end{lem}

\begin{proof} Recalling $\mathcal A \psi(r)=\psi''(r)-\frac 1 r \psi'(r)$, the three first properties directly follow from using and differentiating the ODE. Last, solving \eqref{A-psi-prime} provides the integral representation.
\end{proof}

For a solution $\psi$,  we say that  $\A \psi$ ceases to be nonnegative at $s_0>0$ 
if $(i)$ there is $\ep_1>0$ such that $\A \psi \geq 0$ on $(s_0-\ep_1,s_0)$ and  $(ii)$ for all $\ep_2>0$, there is $s^*\in(s_0,s_0+\ep_2)$ such that $\A \psi (s^*)<0$.

\begin{cor}\label{cor:switch-psi} $\A \psi$ ceases to be nonnegative at $s_0>0$ if and only if $\psi(s_0)$ belongs to $(-\infty,-\alpha)\cup(-\beta,0)\cup(\beta,\alpha)$ and $\psi'(s_0)=-\frac{s_0}{k}g(\psi(s_0))$; and, if so, there is $\ep_3>0$ such that $\A\psi <0$ on $(s_0,s_0+\ep_3)$. 
\end{cor}

\begin{proof}  If $\A\psi$ ceases to be nonnegative at $s_0>0$ then $\A\psi(s_0)=0$ and $(\A\psi)'(s_0)\leq 0$.  We deduce from \eqref{A-psi} and \eqref{A-psi-prime} that $g(\psi(s_0))g'(\psi(s_0))\leq 0$, which means 
$\psi(s_0)\in(-\infty,-\alpha]\cup [-\beta,0]\cup [\beta,\alpha]$. If $\psi(s_0)\in \{-\alpha,0,\alpha\}$, we deduce from  $\psi'(s_0)=-\frac{s_0}{k}g(\psi(s_0))=0$ that $\psi$ is constant,  which is a contradiction.  If $\psi(s_0)=\pm \beta$, we deduce from \eqref{A-psi} and \eqref{A-psi-prime} that $\A\psi(s_0)=(\A\psi)'(s_0)=0$, which enforces $(\A\psi)''(s_0)=0$,  which is contradicted by \eqref{A-psi-prime-prime}. Hence $\psi(s_0)\in(-\infty,-\alpha)\cup (-\beta,0)\cup(\beta,\alpha)$ and $\psi'(s_0)=-\frac{s_0}{k}g(\psi(s_0))$. The converse is obviously true since then $\A\psi (s_0)=0$ and $(\A\psi)'(s_0)<0$.
\end{proof}

We now distinguish the case $k=1$ from the case $k\geq 2$.

\subsubsection{The autonomous case $k=1$}


This case is well understood: apart from being constant, solutions can be unbounded  (case $(i)$ below), heteroclinic orbits or standing waves  (see $(ii)$ below) or periodic  (case $(iii)$ below).  In the sequel, we denote
$$
G(t):=\int_0 ^t g(s)ds.
$$ 

\begin{lem}[Second order ODE Cauchy problem, $k=1$]\label{lem:ODE2-k-egal-1} Let $k=1$. Define the initial energy
\begin{equation}
\label{energy-initial}
E=E(\xi,\theta):=\frac 12 \theta ^2+G(\vert \xi\vert).
\end{equation}
Depending on the initial data $(\xi,\theta)$, the following holds.

\begin{enumerate}
\item [(i)] If $\vert \xi \vert >\alpha$ or $E>G(\alpha)$ then $\psi(r)\to +\infty$ or $\psi(r)\to -\infty$ as $r\to \sup J$.

\item [(ii)] If $\vert \xi\vert \leq \alpha$ and $E=G(\alpha)$ the following three situations may occur: if $\theta=0$ then $\psi\equiv -\alpha$ or $\psi \equiv \alpha$; if $\theta<0$ then $\psi'<0$ on $\R$, $\psi(-\infty)=\alpha$, $\psi(+\infty)=-\alpha$; if $\theta>0$ then $\psi'>0$ on $\R$, $\psi(-\infty)=-\alpha$, $\psi(+\infty)=\alpha$.

\item [(iii)] If $\vert \xi\vert < \alpha$ and $0<E<G(\alpha)$ then the solution $\psi$ is trapped in the interval $(-\alpha,\alpha)$, is $T$ periodic for some $T=T(\xi,\theta)>0$,  
has exactly two zeros on any $[a,a+T)$, has exactly two critical points on any $[a,a+T)$ where it moreover changes monotonicity, is symmetric with respect to any critical point (that is $\psi'(r^*)=0 \Rightarrow \psi(r^*-\cdot)=\psi(r^*+\cdot)$), and is anti-symmetric  with respect to any zero (that is $\psi(r^*)=0\Rightarrow  \psi(r^*-\cdot)=-\psi(r^*+\cdot)$). Moreover $\psi$ is a surjection from $\R$ to $[-M,M]$ where $\vert \xi\vert \leq M <\alpha$ is (uniquely) determined by
$$
\int _{\vert \xi \vert}^{M}g(s)ds =\frac 12 \theta ^2,
$$
and the period is given by
\begin{equation*}
\frac{T}{2}=\int_{-M}^{M}\frac{ds}{\sqrt{2(E-G(s))}}=2\int_{0}^{M}\frac{ds}{\sqrt{2(E-G(s))}}.
\end{equation*}


\item [(iv)] If $\vert \xi\vert \leq \alpha$ and $E=0$, meaning $(\xi,\theta)=(0,0)$, then $\psi\equiv 0$. 
\end{enumerate}
\end{lem}

\begin{proof} The associated energy is conserved, that is
\begin{equation}\label{energy}
\frac 12 {\psi'}^2(r)+G(\psi(r))=E,\quad   \forall r\in J.
\end{equation}
From this and a standard phase plane analysis, see Figure \ref{fig:phase-plane}, we classically reach all the desired conclusions. For instance, in the case $(iii)$, it follows from $E<G(\alpha)$ and \eqref{energy} that $\psi(r)$ cannot touch $\alpha$  nor $-\alpha$. Hence, $\psi(r)$ is trapped in $(-\alpha,\alpha)$, $(\psi(r),\psi'(r))$ cannot blow up in finite time, and thus the solution is global. Next, by combining the phase plane analysis and the identity
$$
\psi'(r)=\pm \sqrt{2(E-G(\psi(r))},
$$
we reach the conclusions of $(iii)$. Details are omitted. 
\end{proof}

\begin{figure}
\begin{center}
  \includegraphics[width=8cm]{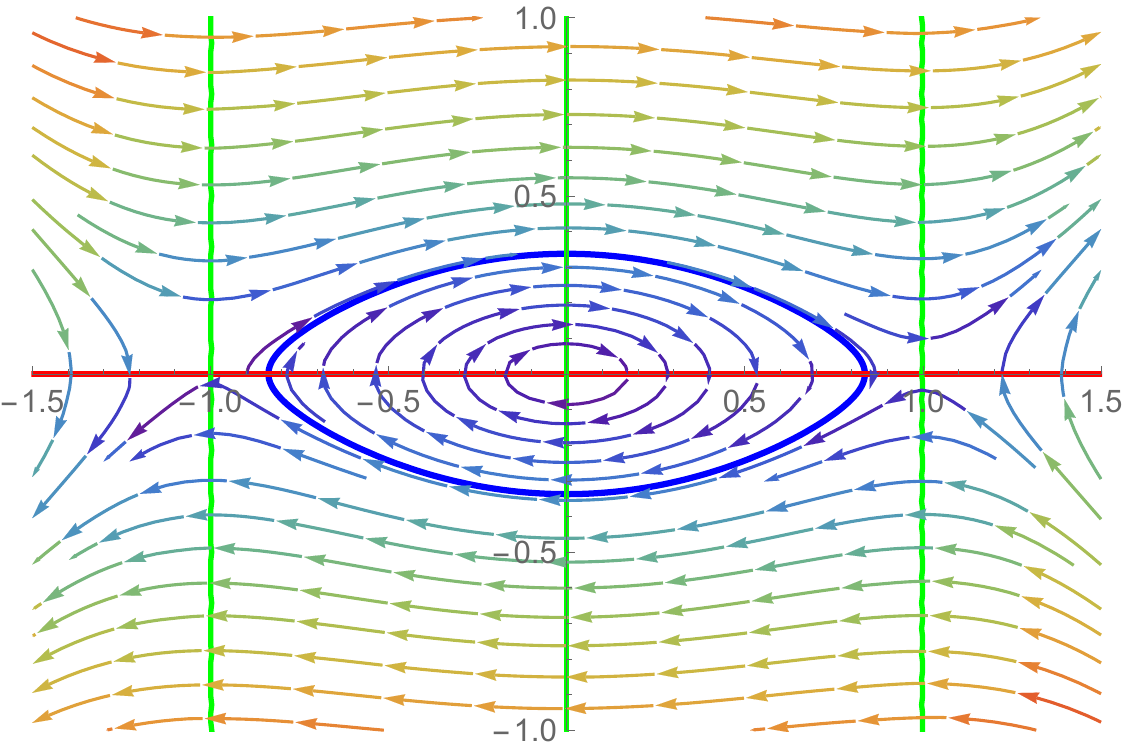}
    \caption{Phase plane analysis for \eqref{cauchy-ordre2} with $k=1$ and the cubic nonlinearity $g(\psi)=\frac 1 4(\psi-\psi^3)$. In red, the nullcline $\psi=0$, in green the nullcline $\psi '=0$, in blue the trajectory starting from $\xi=0.6$, $\theta=-0.2$, providing a periodic solution.}
   \label{fig:phase-plane}
   \end{center}
\end{figure}

\subsubsection{The non autonomous case $k\geq 2$}

The non autonomous case $k\geq 2$ is more tricky. Nevertheless, when $(r_0,\theta)=(0,0)$, the comprehension of the Cauchy problem is still  complete, whatever $k\geq 1$.

\begin{lem}[Second order ODE Cauchy problem, the case  $(r_0,\theta)=(0,0)$]\label{lem:ODE2} Let $k\geq 1$. Assume $(r_0,\theta)=(0,0)$. If $\vert \xi \vert \leq \alpha$ then the solution $\psi$ is global.  Next, depending on the initial data $\xi$, the following holds.
\begin{enumerate}
\item [(i)]  If $\xi\in \{-\alpha,0,+\alpha\}$ then $\psi$ is constant.

\item [(ii)] If $0<\vert \xi\vert <\alpha$ then $\psi$ is oscillating, satisfies $\Vert \psi \Vert_{L^\infty(\R)}=\xi$,
$$
{\psi '}^2(r)\leq 2 \int _{\psi (r)}^\xi g(t)dt\leq 2 G(\xi), \quad \forall r\in\R,
$$
and
$$
\frac{c}{r^{\frac{k-1}{2}}}\leq \vert \psi(r)\vert +\vert \psi'(r)\vert +\vert \psi''(r)\vert  \leq \frac{C}{r^{\frac{k-1}{2}}}, \quad \forall r\geq 1,
$$
for some constants $c=c(\xi)>0$, $C=C(\xi)>0$. Moreover, the critical points of $\psi$ on $[0,+\infty)$ form a sequence $0=x_0<x_1<\cdots<x_n \to +\infty$ such that $\psi(x_k)\psi(x_{k+1})<0$.

\item [(iii)]  If $\xi>+\alpha$ then $\psi '>0$ on $J\cap (0,+\infty)$, and $\psi(r)\to +\infty$ as $r\to \sup J$.

If $\xi<-\alpha$ then $\psi '<0$ on $J\cap (0,+\infty)$, and $\psi(r)\to -\infty$ as $r\to \sup J$.
\end{enumerate}
\end{lem}

\begin{proof}  Item $(i)$ is clear, whereas item $(ii)$ is borrowed from \cite{Man-Mon-Pel-17}, more precisely Proposition 2.3 and Lemma 2.5 together with their proofs. 
 









Let us prove $(iii)$. Since $g$ is odd, we check that $\psi(r;0,-\xi,0)=-\psi(r;0,\xi,0)$, so we  only need to consider  $\xi >\alpha$. From the ODE we have $\psi''(0)=-\frac{g(\xi)}{k}>0$, so that there is $\ep>0$ such that $\psi'>0$ on $(0,\ep)$. 

Assume by contradiction the existence of $r_0>0$ such that $\psi'(r_0)=0$. We may assume that $r_0$ is the smallest positive value with this property so that $\psi(r_0)>\xi>\alpha$ and $\psi''(r_0)\leq 0$. Testing the ODE at $r=r_0$ we reach $\psi''(r_0)>0$, a contradiction. As a result $\psi'>0$ on $J\cap (0,+\infty)$. 

Hence, if $\sup J<+\infty$ then one must have $\psi(r)\to +\infty$ as $r\to \sup J$. On the other hand, if $\sup J=+\infty$, let us assume by contradiction that the limit $\ell>\xi>\alpha$ of $\psi(r)$ in $+\infty$ is finite. From \eqref{psi-prime} we deduce $\psi'(r)\sim \frac{-g(\ell)}{k}r$ as $r\to +\infty$, which contradicts $\ell<+\infty$. As a result, $\ell=+\infty$ and we are done with $(iii)$.
\end{proof}

Last, we  need a comprehension of the  non autonomous case $k\geq 2$, starting from $r_0>0$. When $k\geq 2$, oscillations are typically dampened and some solutions, not trapped in the $k=1$ case, become trapped, and thus oscillating, when $k\geq 2$. Note that many results exist \cite{Mai-98}, \cite{Cab-Eng-Gad-09}, \cite{Har-Jen-13}, when the potential $G$ is convex and/or coercive. In the present bistable situation, the potential $G$ is neither convex nor coercive and, therefore, the issue of blow-up solutions is intricate. Therefore, we believe  Theorem \ref{th:ODE2-bis} and its proof are interesting by themselves in the framework of ODEs.

For our purpose in Section \ref{s:radial},  we only need to consider the one-parameter family of solutions $\psi(r;r_0,\xi,\theta)$ with $r_0>0$, $\xi=\beta$, $\theta=-\frac{g(\beta)}{k}r_0$, which corresponds to a switch from the first order ODE to the second order one for the radial solutions of  \eqref{eq-plus}, see Section \ref{s:radial}.

\begin{theo}[Second order ODE Cauchy problem, the case $r_0>0$]\label{th:ODE2-bis} Let $k\geq 2$. For $r_0>0$, let us denote $\psi(r):=\psi\left(r;r_0,\beta,-\frac{g(\beta)}{k}r_0\right)$. Then (recall that $G(t)=\int_0^t g(s)ds$) there are positive real numbers $\underline{r_0}$, $r_0^{**}$, $\overline{r_0}$ satisfying 
\begin{equation}\label{tous-les-r}
\frac{k\sqrt{2}}{g(\beta)}\left( G(\alpha)-\frac{1}k G(\beta) \right) ^{\frac{1}{2}}\leq \underline{r_0}\leq r_0^{**}\leq \overline{r_0}\leq  \frac{k\sqrt{2}}{g(\beta)}\left( G(\alpha)-G(\beta)+\frac{k-1}k g(\beta)(\alpha+\beta) \right) ^{\frac{1}{2}},
\end{equation}
and such that the following holds.
\begin{enumerate}
\item [(i)] If $0<r_0 <\underline{r_0}$ then, for $r>r_0$, the solution $\psi$ is trapped in $(-\alpha,\alpha)$, global, oscillating and localized. Moreover, there is a sequence $r_0<r_1<\cdots<r_n\to +\infty$ such that, for any $k\in \N$, $\psi$ is strictly monotone on $(r_k,r_{k+1})$, and $\psi(r_{2k})\psi(r_{2k+1})<0$. 
\item [(ii)] If $r_0=r_0^{**}$ then $\psi '<0$ on $(r_0,+\infty)$ and $\psi(+\infty)=-\alpha$. 
\item [(iii)] If $r_0>\overline{r_0}$ then $\psi '<0$ on $J\cap (r_0,+\infty)$, and $\psi(r) \to -\infty$ as $r\to \sup J$.
\end{enumerate}

\end{theo}

\begin{rem}[A  conjecture] \label{rem:not-exhaustive}
As already emphasized above, such an analysis does not stand in the classical frameworks, since the potential $G$ is neither convex nor coercive. Another important difficulty comes from the fact that we look at a one-parameter family: $r_0$ determines both the starting point (at $r=r_0$) and the starting slope ($\psi'(r_0)=-\frac{g(\beta)}{k}r_0$). Because of that, the description in Theorem \ref{th:ODE2-bis} is not exhaustive. In the proof below, we conjecture that the set $B$ is actually  a singleton. In other words, we conjecture that, in the setting of Theorem \ref{th:ODE2-bis}, we may take $\underline{r_0}=r_0^{**}=\overline{r_0}$.
\end{rem}

\begin{proof}  We temporarily denote $\psi(r)=\psi\left(r;r_0,\beta,-\frac{g(\beta)}{k}r_0,g\right)$ to remind the dependence on the nonlinearity $g$. We observe that $\psi(r)=v\left(\frac{r}{r_0}\right)$, where
$$
v(r)=\psi\left(r;1,\beta,-\frac{g(\beta)}{k}r_0^2,r_0^2 g\right).
$$
In the sequel, in order to fix the initial point as 1, we work on $v$ and denote $J_v$ its maximal interval of validity.

If there is a (first) point $r^*>1$  where $v'(r^*)=0$ then, from the equation, $v(r^*)\in(-\alpha,0)$. Then, using the equation and the energy identity 
\begin{equation}\label{energy-decrease}
\frac 1 2 {v'}^2(r)+r_0^2G(v(r))=\frac 12{v'}^2(t)+r_0^2G(v(t))-(k-1)\int _t ^r \frac{{v'}^2(s)}{s}ds,
\end{equation}
we easily see that $v$ is trapped in $(-\alpha,\alpha)$, global, oscillating and localized, which corresponds to case $(i)$. On the other hand if $v'<0$, note that the situation $v(+\infty)=0$ is excluded by 
a Sturm separation argument  similar to \cite[Lemma 2.5]{Man-Mon-Pel-17}. As a result, we  have the partition $(0,+\infty)=A\cup B \cup C$ where
\begin{eqnarray*}
A&:=&\{r_0>0: \exists r^*>1, v'(r^*)=0 \text{ and } v(+\infty)=0\}\\
B&:=&\{r_0>0: v' <0 \text{ on } [1,+\infty) \text{ and } v(+\infty)=-\alpha\}\\
C&:=&\{r_0>0:v' <0 \text{ on } J_v\cap [1,+\infty) \text{ and } v(r)\to -\infty \text{ as } r\nearrow \sup J_v\}.
\end{eqnarray*}

From the continuity of solutions with respect to initial conditions and parameters, both the sets $A$ and $C$ are open.

Now, returning to $\psi(r)$ we naturally have, see Figure \ref{fig:A-B-C},
\begin{eqnarray*}
A&=&\{r_0>0: \exists r^*>r_0, \psi'(r^*)=0 \text{ and } \psi(+\infty)=0\}\\
B&=&\{r_0>0: \psi' <0 \text{ on } [r_0,+\infty) \text{ and } \psi(+\infty)=-\alpha\}\\
C&=&\{r_0>0:\psi' <0 \text{ on } J_\psi\cap [r_0,+\infty) \text{ and } \psi(r)\to -\infty \text{ as } r\nearrow \sup J_\psi\},
\end{eqnarray*}
with obvious notations.

\begin{figure}
\begin{center}
  \includegraphics[width=10cm]{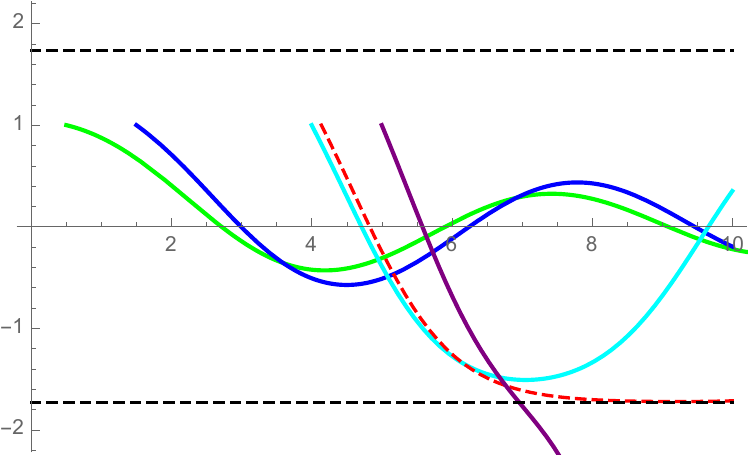}
    \caption{ The solutions $\psi\left(r;r_0,\beta,-\frac{g(\beta)}{k}r_0\right)$ with $k=2$, $g(u):=u-\frac 1 3 u^3$ (so that $\beta=1$ and $\alpha=\sqrt 3$). Solutions of the type $A$: in green $r_0=0.5$, in blue $r_0=1.5$, in cyan $r_0=4$. Solution of the type $B$: in dashed red $r_0\approx 4.139$.  Solutions of the type $C$: in purple $r_0=5$.}
   \label{fig:A-B-C}
   \end{center}
\end{figure}

Here are a few  observations. First, from Lemma \ref{lem:monitoring}, we have $\A\psi(r_0)=(\A\psi)'(r_0)=0$ while $(\A\psi)''(r_0)>0$, so that 
\begin{equation}\label{bidule}
\A\psi > 0 \text{ on } (r_0,r_0+\ep),
\end{equation}
 for some $\ep>0$ (and this is true for any $k\geq 1$). Next, the difference of energy $E(r):=\frac{1}{2}\psi'^2(r)+G(\psi(r))$ between points $t$ and $r>t$ is given by
\begin{equation}
\label{diff-energy}
E(r)-E(t)=-(k-1)\int_{t}^r \frac{\psi'(s)}{s}\psi'(s)ds.
\end{equation}
Notice also that $\frac{d}{dr} \left(\frac{\psi'r)}{r}\right)=\frac 1 r \A \psi (r)$. 
It follows that, if we know that $\A\psi\geq 0$ and $\psi'< 0$ on $(r_0,r)$, then we have the lower bound
\begin{equation}
\label{diff-energy-bis}
E(r)-E(r_0)\geq  -(k-1)\int_{r_0}^r \frac{\psi'(r_0)}{r_0}\psi'(s)ds=\frac{k-1}{k}g(\beta)(\psi(r)-\beta).
\end{equation}
Building on this, we can pursue and conclude.

	$\bullet$ First, if $r_0$ is small then the initial energy is small and the damping is large, so that solutions are expected to be trapped and localized. Precisely, let us prove 
\begin{equation}\label{claim:localized} \text{ there is }\;  \underline{r_0}\geq r_0^{loc}:=\frac{k\sqrt{2}}{g(\beta)}\left( G(\alpha)-\frac{1}k G(\beta) \right) ^{\frac{1}{2}}>0\; \text{ such that }\;   (0,\underline{r_0})\subset A.
\end{equation}
From \eqref{bidule} and Corollary \ref{cor:switch-psi}, we know that $\A\psi \geq 0$ on $(r_0,r_1)$ where $r_1>r_0$ is the first point where $\psi$ values 0. From the equation for $\psi$, we thus have  $\frac{\psi'(r)}{r}\leq -\frac 1 k g(\psi(r))$ on $(r_0,r_1)$, so that \eqref{diff-energy} yields
\begin{equation}
\label{diff-energy-bis-bis}
E(r_1)-E(r_0)\leq  \frac{k-1}k\int_{r_0}^{r_1} g(\psi(s))\psi'(s)ds=\frac{k-1}{k}(G(\psi(r_1))-G(\beta))=-\frac{k-1}{k}G(\beta).
\end{equation}
Now, if $r_0\in C$, then $\psi'<0$ on $(r_1,r_2)$ where $r_2>r_1$ is the point where $\psi(r_2)=-\alpha$, and \eqref{diff-energy} enforces $E(r_1)\geq E(r_2)\geq G(-\alpha)=G(\alpha)$. Hence, from \eqref{diff-energy-bis-bis}, $G(\alpha)\leq \frac{1}{2}\frac{g^2(\beta)}{k^2}r_0^2+G(\beta)-\frac{k-1}{k}G(\beta)$, that is $r_0\geq r_0^{loc}$. If $r_0\in B$, we find the same estimate $r_0\geq r_0^{loc}$ (roughly speaking $r_2=+\infty$). This proves that $(0,r_0^{loc})\subset A$ and thus \eqref{claim:localized}.

$\bullet$ Second, if $r_0$ is large then not only the damping is small but also the initial speed is very negative, so that solutions are expected to be non trapped. Precisely, let us prove that
\begin{equation}\label{claim:non-trapped} \text{ there is }\;  \overline{r_0}\leq \frac{k\sqrt{2}}{g(\beta)}\left( G(\alpha)-G(\beta)+\frac{k-1}k g(\beta)(\alpha+\beta) \right) ^{\frac{1}{2}} \; \text{ such that }\;   (\overline{r_0},+\infty)\subset C.
\end{equation}
If $r_0\in A$, then there is $r_1>r_0$ such that $\psi '<0$ on $(r_0,r_1)$, $\psi'(r_1)=0$, and $\psi (r_1)\in (-\alpha,0)$; furthermore \eqref{bidule} and  Corollary \ref{cor:switch-psi} enforce $\A\psi \geq 0$ on $(r_0,r_1)$. We can thus apply \eqref{diff-energy-bis} and obtain
$$
G(\alpha)-E(r_0)\geq E(r_1)-E(r_0)\geq \frac{k-1}{k}g(\beta)(\psi(r_1)-\beta)\geq \frac{k-1}{k}g(\beta)(-\alpha-\beta),
$$
which means $r_0\leq r_0^{non-tra}:=\frac{k\sqrt{2}}{g(\beta)}\left( G(\alpha)-G(\beta)+\frac{k-1}k g(\beta)(\alpha+\beta) \right) ^{\frac{1}{2}} $. If $r_0\in B$, we find the same estimate $r_0\leq r_0^{non-tra}$ (roughly speaking $r_1=+\infty$). This proves that $(r_0^{non-tra},0)\subset C$ and thus \eqref{claim:non-trapped}. 

$\bullet$ Hence,  both $A$ and $C$ are open and non empty so that $(0,+\infty)=A\cup B\cup C$ enforces $B$ to be non empty, which completes the proof.
\end{proof}

\section{Radial solutions}\label{s:radial}

In this section, we will prove Theorem \ref{th:radial-k-1} ($k=1$) and Theorem \ref{th:radial-k-2} ($k\geq 2$). Recalling Definition \ref{def:radial}, we thus look after a $0<R\leq +\infty$ and a piecewise $\mathcal C^2$ function $u:[0,R)\to \R$, with $u'(0)=0$,  starting from $u(0)=\xi\in \R$, and such that $U(x):=u(\vert x\vert)$ solves \eqref{eq-plus} on $B(0,R)$.  From subsection \ref{ss:radial}, we understood that, depending on the sign of
$$
\A u(r):=u''(r)-\frac 1 ru'(r),
$$
we are facing either a first order ODE studied in subsection \ref{ss:ode-ordre-1} or a second order ODE studied in subsection \ref{ss:ode-ordre-2}. We call {\it switching points}  the points where $\A u(r)$  changes sign, possibly with a jump. If such points exist, in order to construct radial solutions, we need to glue some solutions to \eqref{cauchy-ordre1} with some solutions to \eqref{cauchy-ordre2}. It is therefore crucial to determine when a switch occurs at $r_0>0$ together with the \lq\lq starting behaviour'' at $r_0=0$.

In what follows FOE stands for the First Order Equation (whose solutions are typically denoted $\varphi$) and SOE for the Second Order Equation (whose solutions are typically denoted $\psi$).

\begin{lem}[Switching points]\label{lem:switching} Let $(R,u)$ be a solution to \eqref{eq-plus} in the sense of Definition \ref{def:radial}. 
		\begin{enumerate}
		\item [(i)] It switches from FOE to SOE at $r_0\in(0,R)$ if and only if $\varphi(r_0)=\beta$. The solution is then of class $\mathcal{C}^2$ in a neighbourhood of $r_0$.
		\item [(ii)] It switches from SOE to FOE at $r_0\in(0,R)$ if and only if $\psi(r_0)\in  (-\infty,-\alpha)\cup(-\beta,0)\cup(\beta,\alpha)$ and $ \psi'(r_0)=-\frac{r_0}{k}g(\psi(r_0))$. The solution is then of class $\mathcal{C}^1$ in a neighbourhood of $r_0$, but $u ''(r_0^-)>u''(r_0^+)$. 
	\end{enumerate}
\end{lem}

\begin{proof} Let us prove $(i)$. From Corollary \ref{cor:switch-phi}, the condition $\varphi(r_0)=\beta$ is necessary for the switch to occur. Conversely, $\varphi(r_0)=\beta$ enforces $\varphi'(r_0)=-\frac{r_0}{k}g(\beta)$ and, $u$ being $\mathcal C^1$ has to be equal to $\psi=\psi\left(\cdot;r_0,\beta,-\frac{g(\beta)}{k}r_0\right)$, provided by Theorem \ref{th:ODE2-bis}, on  $(r_0,r_1)$ for some $r_1>0$. It follows from \eqref{bidule}   that the switch has occurred at $r_0$. Last, the continuity of $\A u$ implies that of $u''$ and the solution $u$ is $\mathcal{C}^2$ in a neighbourhood of $r_0$.

Let us prove $(ii)$. From Corollary \ref{cor:switch-psi}, the stated conditions are necessary. Conversely, $u$ being $\mathcal C^1$ has to be equal to  $\varphi=\varphi(\cdot;r_0,\psi(r_0))$, defined at the beginning of subsection \ref{ss:ode-ordre-1}, on $(r_0,r_1)$ for some $r_1>0$. It follows from \eqref{A-phi} that $\A\varphi(r_0)=\frac{r_0^2}{k^2}g(\varphi(r_0))g'(\varphi(r_0))<0$, and thus the switch has occurred at $r_0$. Last, $\A\psi(r_0)=0>\A \varphi (r_0)$ transfers into  $u''(r_0^-)>u''(r_0^+)$.
\end{proof}

\begin{lem}[When starting]\label{lem:starting} Let $(R,u)$ be a solution to \eqref{eq-plus} in the sense of Definition \ref{def:radial}.  Denote $u(0)=\xi\in\R$.
	\begin{enumerate}
	\item [(i)] If $\xi\in \{-\alpha,0,\alpha\}$, then $R=+\infty$ and $u$ is constant. 
		\item [(ii)] If $\xi\in (-\infty,-\alpha)\cup[-\beta,0)\cup(\beta,\alpha)$, then there is $\ep>0$ such that $\A u< 0$ on $(0,\ep)$, so that $u$ follows the FOE on $[0,\ep)$. 
		\item [(iii)] if $\xi\in (-\alpha,-\beta)\cup(0,\beta]\cup(\alpha,+\infty)$, then there is $\ep>0$ such that $\A u> 0$ on $(0,\ep)$, so that $u$ follows the SOE on $[0,\ep)$. 
	\end{enumerate}
\end{lem}

\begin{proof} A key observation is that if $u$ follows the FOE then $u'=-\frac{r}{k}g(u)$, while if $u$ follows the SOE then $\A u\geq 0$ so that $u'\leq -\frac{r}{k}g(u)$. In particular if $r$ is such that $u(r)=\beta$ then $u'(r)\leq -\frac{r}{k}g(\beta)<0$, i.e.  $u=\beta$ is a wall that can be crossed in the decreasing direction only. As a result it follows from Lemma \ref{lem:switching} $(i)$ that there can be at most one switch from FOE to SOE, and the total number of switches is at most three. In particular 
\begin{equation}
\label{starting-signe}
\text{ there is $\ep>0$ such that $\A u$ has a constant sign on $(0,\ep)$.}
\end{equation}
 Conclusion $(i)$ is then obvious and we now assume $\xi\not\in\{-\alpha,0,\alpha\}$. 

Assume $\A u\leq 0$ on $[0,\ep)$ so that $u=\varphi$ follows the FOE. According to \eqref{A-phi}, $\varphi(r)$ lies in the region  $\{gg'\leq 0\}$, hence $\xi\in (-\infty,-\alpha)\cup[-\beta,0)\cup[\beta,\alpha)$. But if $\xi=\beta$ then, from the equation for $\varphi$, $\varphi'(0)=0$, $\varphi''(0)=-\frac 1 k g(\beta)<0$ so that $u=\varphi$  enters the interval $(0,\beta)$ where it cannot follow the FOE. Hence $\xi\in (-\infty,-\alpha)\cup[-\beta,0)\cup(\beta,\alpha)$, which proves $(iii)$, up to the strict inequality for $\A u$ that is easily checked via \eqref{A-phi}, \eqref{A-phi-prime}.

Assume $\A u\geq 0$ on $[0,\ep)$ so that $u=\psi$ follows the SOE. From \eqref{A-psi-integrated}, we get $\ds\A\psi(r)=-\int_0^r\frac{s^k}{r^k}\psi'(s)g'(\psi(s))ds$ on $[0,\ep)$. Taking into account $\psi''(0)=-\frac{1}{k}g(\xi)$, we infer that, up to reducing $\ep>0$, $-\psi'$ has the same sign as $g(\psi)$ on $(0,\ep)$. Hence, the sign of $-\psi'(s)g'(\psi(s))$ is the one of $g(\psi(s))g'(\psi(s))$   which implies that $\psi(s)$ lives in the region $\{gg'\geq0\}$, hence $\xi\in (-\alpha,-\beta]\cup(0,\beta]\cup(\alpha,+\infty)$. It remains to exclude the case $\xi=-\beta$: in this case $\psi''(0)>0$ so that, up to reducing $\ep>0$, $\psi'(r)>0$ on this interval; but $\A\psi(r)\geq 0$ implies $(\A\psi)'(r)=-\frac{k}{r}\A\psi(r)-\psi'(r)g'(\psi(r))<0$, a contradiction. Hence $\xi\in(-\alpha,-\beta)\cup(0,\beta]\cup(\alpha,+\infty)$, which proves $(ii)$, up to the strict inequality for $\A u$ that is easily checked via \eqref{A-psi}, \eqref{A-psi-prime}, \eqref{A-psi-prime-prime}.
\end{proof}

\subsection{Starting with $\A u> 0$}\label{ss:starting-plus}

In this subsection we assume $u(0)=\xi$ with $\xi\in (-\alpha,-\beta)\cup(0,\beta]\cup(\alpha,+\infty)$. We know from Lemma \ref{lem:starting} that 
\begin{equation}
\label{starting-plus}
\exists \ep >0, \forall r\in (0,\ep), \mathcal A u(r)> 0,
\end{equation}
 and that the solution $u$ follows the SOE on $[0,\ep)$. We thus denote $\psi_0(r):=\psi(r;0,\xi,0)$ the solution to \eqref{cauchy-ordre2}, that is
\begin{equation}\label{cauchy-ordre2-k-egal-1-bis-bis}
\left\{\begin{array}{lllll} \psi_0''+\frac{k-1}{r}\psi_0'(r)+g(\psi_0)&=&0,\\
\psi_0(0) &=& \xi,\\
\psi'_0(0)&=&0,
\end{array}\right.
\end{equation}
 defined on some open interval $J_0$ containing $0$, and we have to start, at least on $(0,\ep_0:=\min(\ep, \sup J_0))$, with $u\equiv \psi_0$.  We know that there is $\eta>0$ such that $\A \psi_0>0$ in $(0,\eta)$, and we can define
 \begin{equation}
 	\label{def:t0}
 	t_0:=\sup\{t>0: \A \psi_0> 0 \text{ on } (0,t)\} \in (0,\sup J_0].
 \end{equation}
 
Now we distinguish some cases depending on the initial value $\xi$.

$\bullet$ If $\xi>\alpha$ we know 
from Lemma \ref{lem:ODE2} $(iii)$ that $\psi_0'(t_0)>0$ and $\psi_0(t_0)>\alpha$ so that, according to Lemma \ref{lem:switching}, no switch can occur. Hence $t_0=\sup J_0$ and, in view of Lemma \ref{lem:ODE2} $(iii)$,  $u\equiv \psi_0$  provides a positive, increasing and unbounded solution on $[0,\sup J_0)$.

$\bullet$ If $-\alpha<\xi< -\beta$, it follows from Lemma \ref{lem:ODE2} $(ii)$ that $\psi_0$ is oscillating (periodic when $k=1$).
From  $\psi_0''(0)=-\frac{g(\xi)}k>0$ and thus, from Lemma \ref{lem:ODE2} $(ii)$, there is $d>0$ such that $\psi_0'>0$ on $(0,d]$ and $\psi_0(d)=0$. Therefore, from the equation, $\psi_0''(d)\leq 0$ and then $\A \psi_0(d)=\psi_0''(d)-\frac{1}{d}\psi_0'(d)<0$. As a result $t_0<d$ and from Lemma \ref{lem:switching} we obtain $\psi_0(t_0)\in (-\beta,0)$. 

Now, after the switching point $t_0$, we consider $\varphi_0(r):=\varphi(r;t_0,\psi_0(t_0))$ the solution to \eqref{cauchy-ordre1}. From Lemma \ref{lem:ODE1} and \eqref{A-phi}, we know that $\A \varphi _0<0$ on $(t_0,+\infty)$, and there is no more switching point.

Gluing $\psi _0$ and $\varphi_0$ provides a negative, increasing and localized solution $u$ on $[0,+\infty)$.

$\bullet$ If $0<\xi \leq  \beta$,  we know $\psi_0''(0)=-\frac{g(\xi)}k<0$ and thus, from Lemma \ref{lem:ODE2}, there are  $0<b<d$ such that $\psi_0'<0$ on $(0,b)$, $\psi_0(b)\in [-\xi,0)$, $\psi_0'>0$ on $(b,d]$, $\psi_0(d)=0$. From the equation, $\psi_0''(d)<0$ and then $\A \psi_0(d)=\psi_0''(d)-\frac{1}{d}\psi_0'(d)<0$, so that $t_0<d$ and even  $b<t_0<d$ as one can check via Corollary \ref{cor:switch-psi} and the equation for $\psi_0$. In particular $\psi_0(t_0)\in (-\xi,0)\subset (-\beta,0)$.

Now, after the switching point $t_0$, we consider $\varphi_0(r):=\varphi(r;t_0,\psi_0(t_0))$ the solution to \eqref{cauchy-ordre1}. From Lemma \ref{lem:ODE1} and \eqref{A-phi}, we know that $\A \varphi _0<0$ on $(t_0,+\infty)$, and there is no more switching point.

As  a result, gluing $\psi _0$ and $\varphi_0$ provides a  sign-changing, decreasing-increasing and localized solution $u$ on $[0,+\infty)$.

\subsection{Starting with $\A u< 0$}\label{ss:starting-moins}

In this subsection we assume $u(0)=\xi$ with  $\xi\in (-\infty,-\alpha)\cup[-\beta,0)\cup(\beta,\alpha)$. We know from Lemma \ref{lem:starting} that 
\begin{equation}
\label{starting-moins}
\exists \ep >0, \forall r\in (0,\ep), \mathcal A u(r)< 0,
\end{equation}
 and that the solution $u$ follows the FOE on $[0,\ep)$. We thus denote  $\varphi_0(r):=\varphi(r;0,\xi)$ the solution to \eqref{cauchy-ordre1}, defined on some open interval $I_0$ containing $0$, and we have to start, at least on $(0,\ep_0:=\min(\ep, \sup I_0))$, with $u\equiv \varphi_0$.
 
 Now we distinguish some cases depending on the initial value $\xi$.

$\bullet$ If $\xi <-\alpha$ then \eqref{A-phi}, combined with Lemma \ref{lem:ODE1} and Assumption \ref{ass:nonlinearity}, shows that $\A \varphi_0 (r)< 0$ for all $r\in (0,\sup I_0)$. As a result, there is no switching point and $u\equiv \varphi_0$ as long as it exists. This provides a negative, decreasing and unbounded solution $u$ on $[0,\sup I_0)$. 

$\bullet$ If $-\beta\leq \xi<0$ then \eqref{A-phi}, combined with Lemma \ref{lem:ODE1} and Assumption \ref{ass:nonlinearity}, shows that $\A \varphi_0 (r)< 0$ for all $r\in (0,+\infty)$. As a result, there is no switching point and $u\equiv \varphi_0$ on $[0,+\infty)$. This provides a negative, increasing and localized solution $u$ on $[0,+\infty)$.  Observe  that such solutions \lq\lq without any switch'' were already described in \cite{Bir-Gal-20}.

$\bullet$ We now focus on the case $\beta <\xi <\alpha$, which is the richest. From item $(ii)$ of Lemma \ref{lem:ODE1}, we know that $\varphi_0'<0$  in $(0,+\infty)$. In particular, from \eqref{A-phi} and Assumption \ref{ass:nonlinearity}, we have $\A \varphi_0 <0$ as long as $\varphi_0$ has not reached $\beta$, which happens at some $r_0>0$. From the expression \eqref{varphi} of the solution $\varphi_0$, we know that $r_0>0$ is given by 
\begin{equation}\label{r-zero}
r_0^{2}=2 k\int _\beta ^\xi \frac{ds}{g(s)}.
\end{equation}
We thus have to select
\begin{equation}\label{debut}
	u\equiv \varphi _0 \quad \text{ on }\; [0,r_0],
\end{equation}
and from Lemma \ref{lem:switching} we know that $u$ switches from FOE to SOE at $r_0$.

From now we distinguish the situation $k=1$ (for which we build on Lemma \ref{lem:ODE2-k-egal-1}) from the situation $k\geq 2$ (for which we build on Theorem \ref{th:ODE2-bis}).

\subsubsection{Switch at $r=r_0$ when $k=1$} 

We now consider $\psi_0(r):=\psi(r;r_0,\beta,-r_0g(\beta))$ the solution to \eqref{cauchy-ordre2} with $k=1$, that is
\begin{equation}\label{cauchy-ordre2-k-egal-1-bis}
\left\{\begin{array}{lllll} \psi_0''+g(\psi_0)&=&0,\\
\psi_0(r_0) &=& \beta,\\
\psi'_0(r_0)&=&-r_0g(\beta),
\end{array}\right.
\end{equation}
defined on some open interval $J_0$ containing $r_0$. From Lemma \ref{lem:ODE2-k-egal-1}, the associated energy given by
\begin{equation}
E_0=\frac 12 (r_0g(\beta))^{2}+G(\beta)=g^2(\beta)\int _\beta ^\xi \frac{ds}{g(s)}+\int_0^\beta g(s)ds
\end{equation}
determines the outcome of $\psi _0$. From \eqref{bidule}, there is $\eta >0$ such that $\A \psi_0>0$ in $(r_0,r_0+\eta)$, and we can define
\begin{equation}
\label{def:t1}
t_1:=\sup\{t>r_0: \A \psi_0>0 \text{ on } (r_0,t)\} \in (r_0,\sup J_0].
\end{equation}
Recalling the definition of $\xi ^{*}$ in \eqref{def:xi-etoile}, we now distinguish three regimes for the initial data $\xi$.

$\blacktriangleright$ Assume $\xi ^*<\xi<\alpha$.  Then $E_0>G(\alpha)$ and it follows from Lemma \ref{lem:ODE2-k-egal-1} and the associated phase plane analysis that we have $\psi_0 '<0$, and thus $\psi_0<\beta$, on $(r_0,\sup J_0)$.

If $t_1=\sup J_0$, there is no more switching and we have to select
\begin{equation}\label{suite}
u\equiv \psi_0 \quad \text{ on }\; [r_0,\sup J_0).
\end{equation}
Gluing \eqref{debut} and \eqref{suite} provides a sign-changing, decreasing and unbounded solution $u$ on $[0,\sup J_0)$. 

On the other hand,  if $t_1<\sup J_0$ then $t_1$ is a switching point. In virtue of Lemma \ref{lem:switching} $(ii)$ we have   $\psi_0(t_1)<-\alpha$. After the new switching $t_1$,  we thus consider $\varphi_1(r):=\varphi(r;t_1,\psi_0(t_1))$ the solution to \eqref{cauchy-ordre1} on some open interval $I_1$, and we are back to a situation studied above. Gluing $\varphi_0$, $\psi_0$ and $\varphi_1$ provides again a sign-changing, decreasing and unbounded solution $u$ on $[0,\sup I_1)$.

$\blacktriangleright$ Assume $\xi=\xi^{*}$. Then $E_0=G(\alpha)$ and it follows from Lemma \ref{lem:ODE2-k-egal-1} that we have $\psi_0 '<0$, and thus $\psi_0<\beta$ on $(r_0,+\infty)$, and $\psi_0(+\infty)=-\alpha$. From Lemma \ref{lem:switching} $(ii)$ there is no more switching and we have to select 
\begin{equation}
u\equiv \psi _0 \quad \text{ on }\; [r_0,+\infty).
\end{equation}
This provides a sign-changing, decreasing and bounded solution with $u(+\infty)=-\alpha$.

$\blacktriangleright$ Assume $\beta<\xi<\xi^{*}$. Then $E_0<G(\alpha)$ and it follows from Lemma \ref{lem:ODE2-k-egal-1} that $\psi_0 $ is periodic. More precisely from  Lemma \ref{lem:ODE2-k-egal-1} there are $r_0<a<b<c<d$ such that $\psi_0'<0$ on $(r_0,b)$, $\psi_0(a)=-\beta$, $\psi_0(b)=-M_0\in(-\alpha,-\beta)$, $\psi_0'>0$ on $(b,d]$, $\psi_0(c)=-\beta$, $\psi_0(d)=0$ (see Figure \ref{fig:cartoon}).  In particular $\A\psi_0(d)=-\frac 1 d \psi_0 '(d)<0$ so that $t_1<d$. By definition of $t_1$ we have $\A\psi (t_1)=0$, $(\A\psi)'(t_1)\leq 0$. Using again \eqref{A-psi}, \eqref{A-psi-prime}, and \eqref{A-psi-prime-prime}, we see that $(\A\psi)'(t_1)=0$ would imply $(\A \psi)''(t_1)>0$ which is a contradiction. Hence $(\A\psi)'(t_1)< 0$ and $t_1$ is a switching point so that, according to Lemma \ref{lem:switching} $(ii)$, $\psi_0(t_1)\in(-\beta,0)$ with $\psi_0'(t_1)>0$ that is $c<t_1<d$. 
 We thus select
\begin{equation}\label{switch1}
u\equiv \psi _0 \quad \text{ on }\; [r_0,t_1].
\end{equation}

\begin{figure}
\begin{center}
  \includegraphics[width=10cm]{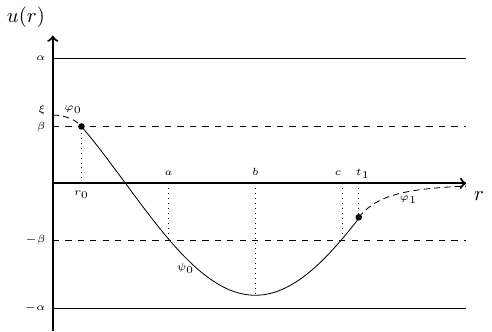}
    \caption{A cartoon of the solution starting from $\beta <\xi<\xi^*$ when $k=1$, and possibly from $\beta<\xi<\underline \xi$ when $k\geq 2$ (in this case $\psi_0(b)\in(-\beta,0)$ is not excluded). The solution starts as $\varphi _0$, switches at $r=r_0$ to $\psi _0$, and switches at $r=t_1$ to $\varphi _1$.}
   \label{fig:cartoon}
   \end{center}
\end{figure}

\noindent {\bf Second switch at $r=t_1$.} We now consider $\varphi_1(r):=\varphi(r;t_1,\psi_0(t_1))$ the solution to \eqref{cauchy-ordre1}, defined on some open interval $I_1$ containing $t_1$. From $-\beta<\psi_0(t_1)<0$ and Lemma \ref{lem:ODE1}, we know that $\sup I_1=+\infty$ and that $\A \varphi _1<0$ on $(r_1,+\infty)$. There is no more switching point and we have to select
\begin{equation}\label{switch2}
u\equiv \varphi _1 \quad \text{ on }\; [t_1,+\infty).
\end{equation}
As a conclusion,  when $\beta <\xi<\xi ^*$, gluing \eqref{debut}, \eqref{switch1} and \eqref{switch2} provides a sign-changing, decreasing-increasing and localized solution, as in Figure \ref{fig:cartoon}.

\medskip

Collecting the above results of this section, we have proved Theorem \ref{th:radial-k-1} on radial solutions when $k=1$. 
\qed

\subsubsection{Switch at $r=r_0$ when $k\geq 2$} \label{ss:switch}

We now consider $\psi_0(r):=\psi(r;r_0,\beta,-\frac{r_0}k g(\beta))$ the solution to \eqref{cauchy-ordre2}, that is
\begin{equation}\label{cauchy-ordre2-encore}
\left\{\begin{array}{lllll} \psi_0''+\frac{k-1}{r}\psi_0'+g(\psi_0)&=&0,\\
\psi_0(r_0) &=& \beta,\\
\psi'_0(r_0)&=&-\frac{r_0}k g(\beta),
\end{array}\right.
\end{equation}
defined on some open interval $J_0$ containing $r_0$.

Exactly as in the $k=1$ case,  we can define
\begin{equation}
\label{def:t1-bis}
t_1:=\sup\{t>r_0: \A \psi_0>0 \text{ on } (r_0,t)\} \in (r_0,\sup J_0].
\end{equation}
We now distinguish three regimes for the initial data $\xi$. Obviously there is a one-to-one relation between $r_0\in(0,+\infty)$ and $\xi\in(\beta,\alpha)$ through \eqref{r-zero}. Hence, $\overline{r_0}\geq r_0^{**}\geq\underline{r_0}$, from \eqref{tous-les-r} in Theorem \ref{th:ODE2-bis}, provide some $\overline{\xi}\geq \xi^{**}\geq \underline{\xi}$ in $(\beta,\alpha)$, which are those appearing in Theorem \ref{th:radial-k-2} item $(v)$. The arguments are then very comparable to the $k=1$ case and they are only sketched.

$\blacktriangleright$ Assume $\overline \xi <\xi<\alpha$, corresponding to  $r_0>\overline{r_0}$ through \eqref{r-zero}.  Then it follows from Theorem \ref{th:ODE2-bis} $(iii)$  that we have $\psi_0 '<0$, and thus $\psi_0<\beta$, on $(r_0,\sup J_0)$.

If $t_1=\sup J_0$, there is no more switching and we have to select
\begin{equation}\label{suite-bis}
u\equiv \psi_0 \quad \text{ on }\; [r_0,\sup J_0).
\end{equation}
Gluing \eqref{debut} and \eqref{suite-bis} provides a sign-changing, decreasing and unbounded solution $u$ on $[0,\sup J_0)$. 

On the other hand,  if $t_1<\sup J_0$ then it is a switching point and $\psi_0(t_1)<-\alpha$. After the new switching $t_1$,  we thus consider $\varphi_1(r):=\varphi(r;t_1,\psi_0(t_1))$ the solution to \eqref{cauchy-ordre1} on some open interval $I_1$, and we are back to a situation studied above. Gluing $\varphi_0$, $\psi_0$ and $\varphi_1$ provides again a sign-changing, decreasing and unbounded solution $u$ on $[0,\sup I_1)$.

$\blacktriangleright$ Assume $\xi=\xi^{**}$ corresponding to $r_0=r_0^{**}$ through \eqref{r-zero}. Then  it follows from Theorem \ref{th:ODE2-bis} $(ii)$
that we have $\psi_0 '<0$, and thus $\psi_0<\beta$, on $(r_0,+\infty)$ and $\psi_0(+\infty)=-\alpha$. From Lemma \ref{lem:switching} $(ii)$, there is no more switching, $t_1=+\infty$, and we have to select 
\begin{equation}
u\equiv \psi _0 \quad \text{ on }\; [r_0,+\infty).
\end{equation}
This provides a sign-changing, decreasing and bounded solution with $u(+\infty)=-\alpha$.  

$\blacktriangleright$ Assume $\beta<\xi<\underline \xi$ corresponding to $0<r_0<\underline{r_0}$ through \eqref{r-zero}. Then  it follows from Theorem \ref{th:ODE2-bis} $(i)$ that $\psi_0 $ is trapped, global, oscillating and localized. In particular,  there are $r_0<b<d$ such that $\psi_0'<0$ on $(r_0,b)$, $\psi_0(b)=-M_0\in(-\alpha,0)$, $\psi_0'>0$ on $(b,d]$, $\psi_0(d)=0$. 
In particular $\A\psi_0(d)=-\frac k d \psi_0 '(d)<0$ so that $t_1<d$. By definition of $t_1$ we have $\A\psi (t_1)=0$, $(\A\psi)'(t_1)\leq 0$. Using again \eqref{A-psi}, \eqref{A-psi-prime}, and \eqref{A-psi-prime-prime}, we see that $(\A\psi)'(t_1)=0$ would imply $(\A \psi)''(t_1)>0$ which is a contradiction. Hence $(\A\psi)'(t_1)< 0$ and $t_1$ is a switching point so that, according to Lemma \ref{lem:switching} $(ii)$, $\psi_0(t_1)\in(-\beta,0)$ with $\psi_0'(t_1)>0$. We thus select
\begin{equation}\label{switch1-bis}
u\equiv \psi _0 \quad \text{ on }\; [r_0,t_1].
\end{equation}

\noindent {\bf Second switch at $r=t_1$.} We now consider $\varphi_1(r):=\varphi(r;t_1,\psi_0(t_1))$ the solution to \eqref{cauchy-ordre1}, defined on some open interval $I_1$ containing $t_1$. From $-\beta<\psi_0(t_1)<0$ and Lemma \ref{lem:ODE1}, we know that $\sup I_1=+\infty$ and that $\A \varphi _1<0$ on $(r_1,+\infty)$. Hence there is no more switching point and we have to select
\begin{equation}\label{switch2-bis}
u\equiv \varphi _1 \quad \text{ on }\; [t_1,+\infty).
\end{equation}
As a conclusion,  when $\beta <\xi<\underline \xi $, gluing \eqref{debut}, \eqref{switch1-bis} and \eqref{switch2-bis} provides a sign-changing, decreasing-increasing and localized solution, possibly as in Figure \ref{fig:cartoon}.  

\medskip

Collecting the above results of this section, we have proved Theorem \ref{th:radial-k-2} on radial solutions when $k\geq 2$. \qed

\subsection{Further  comments}\label{ss:final}

{\bf The role of the shape of $g$.}  Observe that a radial solution always satisfies
$$
u'(r)\leq -\frac{r}{k}g(u(r)).
$$
Indeed, for the FOE this is obviously an equality while, for the SOE, the inequality comes from $\A u\geq 0$. This has two consequences for a nonconstant solution $u$. First,  $u'(r)<0$ for all $r>0$ such that $g(u(r))>0$;  second $u'(r)$ can vanish if $g(u(r))<0$ only and, if so, $u$ follows the SOE and $u''(r)>0$. On the other hand, as revealed above,  the possibility of switching, as well as the choice of the \lq\lq starting'' equation at $r=0$, depends mainly on the sign of the product $gg'$.

Building on this, one may handle more complicated shapes for the nonlinearity $g$. Nevertheless, for the sake of clarity, we restricted ourselves to Assumption \ref{ass:nonlinearity} which already includes the most classical bistable nonlinearities.

\medskip

\noindent {\bf Solutions starting at $r_0>0$.} Obviously, the machinery we have developed to construct radial solutions on  balls $B(0,R)$ can  be applied for domains of the form $\{x\in \R^N: 0<r_0<\vert x\vert < R\leq +\infty\}$.

One may then wonder if those solutions can be extended \lq\lq towards the left'' (at least slightly). Building on \eqref{A-phi}-\eqref{A-phi-prime}, \eqref{A-psi}-\eqref{A-psi-prime}-\eqref{A-psi-prime-prime}, and Lemma \ref{lem:switching}, one can convince oneself (details are omitted) that the solutions that cannot be extended towards the left are $(i)$  the ones starting with the SOE and from $u(r_0)=\xi$, $u'(r_0)= -\frac{r_0}{k}g(\xi)$, where $\xi\in (-\alpha,-\beta)\cup(0,\beta)\cup(\alpha,+\infty)$; $(ii)$  the ones starting with the FOE and from $u(r_0)=-\beta$.

\medskip

\noindent{\bf Overlapping  solutions.} There are switching points $(r_0,\xi)$ that are reached by two solutions, different for $r<r_0$ but equal for $r\geq r_0$. 

Indeed,  consider a decreasing solution of type $(v)$-$(a)$ in Theorems \ref{th:radial-k-1} and \ref{th:radial-k-2}, which 
switches from  SOE  to FOE at some $r_0$ with $u(r_0)<-\alpha$ (one can prove that this always occurs when $g(u)\to +\infty$ as $u\to -\infty$). This solution encounters a solution $\overline u$ that follows the FOE on $[0,r_0)$, of the type $(ii)$ in Theorems \ref{th:radial-k-1} and \ref{th:radial-k-2}, and coincides with $u$ for $r\geq r_0$.

This phenomenon also applies to solutions not defined at $r=0$.

For instance, consider a solution of  type  $(v)$ in Theorems \ref{th:radial-k-1} and \ref{th:radial-k-2}: in particular it switches from FOE to SOE at some $r_0>0$ with $u(r_0)=\beta$, $u'(r_0)=-\frac{r_0}{k}g(\beta)$, see the above constructions and Figure \ref{fig:cartoon}. This solution encounters a solution $\overline u$ that follows the SOE on the left of $r_0$ and coincides with $u$ for $r\geq r_0$. On the left of $r_0$, $\overline u(\equiv \psi)$ is defined either on $[s_0,r_0]$ for some $0< s_0<r_0$ and is not defined for $r<s_0$, or on $(r_0,s_0]$ for some $0\leq s_0<r_0$ and $u(r)\to+\infty$ as $r\to s_0^+$.

Similarly, consider a solution of  type  $(v)$-$(c)$ in Theorems \ref{th:radial-k-1} and \ref{th:radial-k-2}: in particular,  it switches from SOE to FOE at some $t_1>0$ with $u(t_1)\in(-\beta,0)$, $u'(t_1)=-\frac{t_1}{k}g(u(t_1))>0$, see the above constructions and Figure \ref{fig:cartoon}. This solution encounters a solution $\overline u$ that follows the FOE on the left of $t_1$ and coincides with $u$ for $r\geq t_1$. On the left of $t_1$, $\overline u(\equiv \varphi)$ is defined on  $[t_0,t_1]$  for some $0<t_0<t_1$, with $\varphi(t_0)=-\beta$, and is not defined for $r<t_0$.

\bigskip

\noindent{\bf Acknowledgement.} Matthieu Alfaro is supported by  the {\it région Normandie} project BIOMA-NORMAN 21E04343 and the ANR project DEEV ANR-20-CE40-0011-01. 

\bibliographystyle{siam}  

\bibliography{biblio}

\end{document}